\documentclass{amsart}
\usepackage{amsmath, amssymb, amsfonts}
\usepackage[all]{xypic}
\usepackage[pdftex]{graphicx}
\usepackage{enumitem}

\newtheorem{theorem}{Theorem}[section]
\newtheorem{lemma}[theorem]{Lemma}
\newtheorem{prop}[theorem]{Proposition}

\theoremstyle{definition}
\newtheorem{definition}[theorem]{Definition}
\newtheorem{example}[theorem]{Example}

\theoremstyle{remark}
\newtheorem{remark}[theorem]{Remark}

\newcommand{\Z}{\mathbb{Z}}

\newcommand{\Q}{\mathbb{Q}}

\newcommand{\N}{\mathbb{N}}
\newcommand{\F}{\mathbb{F}}
\newcommand{\Ann}{\mathrm{Ann}}

\newcommand{\Gal}{\mathrm{Gal}}
\newcommand{\Tr}{\mathrm{Tr}}

\numberwithin{equation}{section}

\title[Explicit Galois module structure of weakly ramified extensions]{Explicit integral Galois module structure of \\weakly ramified extensions of local fields}

\author{Henri Johnston}
\address{
Department of Mathematics\\
University of Exeter\\
Exeter\\
EX4 4QF\\
U.K.
}
\email{H.Johnston@exeter.ac.uk}
\urladdr{http://emps.exeter.ac.uk/mathematics/staff/hj241}

\date{Version of 15th September 2014}
\begin{document}

\maketitle

\begin{abstract}
Let $L/K$ be a finite Galois extension of complete local fields with finite residue fields and let $G=\Gal(L/K)$.
Let $G_{1}$ and $G_{2}$ be the first and second ramification groups. 
Thus $L/K$ is tamely ramified when $G_{1}$ is trivial and we say that $L/K$ is weakly ramified when $G_{2}$ is trivial. 
Let $\mathcal{O}_{L}$ be the valuation ring of $L$ and let $\mathfrak{P}_{L}$ be its maximal ideal.
We show that if $L/K$ is weakly ramified and $n \equiv 1 \bmod |G_{1}|$ then $\mathfrak{P}_{L}^{n}$ 
is free over the group ring $\mathcal{O}_{K}[G]$, and we construct an explicit generating element.
Under the additional assumption that $L/K$ is wildly ramified, we then show that every free generator of
$\mathfrak{P}_{L}$ over $\mathcal{O}_{K}[G]$ is also a free generator of 
$\mathcal{O}_{L}$ over its associated order in the group algebra $K[G]$.
Along the way, we prove a `splitting lemma' for local fields, which may be of independent interest.
\end{abstract}

\maketitle

%    Text of article.

\section{Introduction}

Let $L/K$ be a finite Galois extension of complete local fields with finite residue fields and let $G=\Gal(L/K)$.
Let $\mathcal{O}_{L}$ be the valuation ring of $L$ and let $\mathfrak{P}_{L}$ be its maximal ideal. 
We recall that for $i \geq -1$ the ramification groups of $L/K$ are
\[
G_{i} := \{ \sigma \in G \mid (\sigma-1)(\mathcal{O}_{L}) \subseteq \mathfrak{P}_{L}^{i+1} \}.
\]
Thus $L/K$ is unramified if and only if $G_{0}$ is trivial and is tamely ramified if and only if $G_{1}$ is trivial.
We say that $L/K$ is weakly ramified if and only if $G_{2}$ is trivial.
In the case that $L/K$ is weakly ramified we shall consider 
the structure of both fractional ideals
$\mathfrak{P}_{L}^{n}$ with $n \equiv 1 \bmod |G_{1}|$ over the group ring $\mathcal{O}_{K}[G]$ 
and of $\mathcal{O}_{L}$ over its associated order
$\mathfrak{A}_{L/K} := \{ x \in K[G] \mid x\mathcal{O}_{L} \subseteq \mathcal{O}_{L} \}$.

A result often attributed to E.\ Noether is that if $L/K$ is tamely ramified then
$\mathcal{O}_{L}$ is free (of rank $1$) as a module over the group ring $\mathcal{O}_{K}[G]$;
in fact as noted in \cite[\S 1]{MR1373957} she only stated and proved the result in the case that the residue characteristic of $K$
does not divide $|G|$ (see \cite{0003.14601}).
Ullom \cite{MR0263790} proved the following: 
$L/K$ is tamely ramified if and only if every non-zero fractional ideal in $L$ is free over $\mathcal{O}_{K}[G]$;
if any non-zero fractional ideal of $L$ is free over $\mathcal{O}_{K}[G]$ then $L/K$ must be weakly ramified;
and if $L/K$ is totally and weakly ramified then $\mathfrak{P}_{L}$ is free over $\mathcal{O}_{K}[G]$.
K\"ock  \cite[Th.\ 1.1]{MR2089083} used cohomological methods to show the more general result that
$\mathfrak{P}_{L}^{n}$ is a free $\mathcal{O}_{K}[G]$-module (of rank $1$) if and only if $L/K$ is weakly ramified and $n \equiv 1 \bmod |G_{1}|$;
this also follows from a minor variant of work of Erez \cite[Th.\ 1]{MR1128708} on the square root of the inverse different
or can be proved by the methods developed in Ullom's papers \cite{MR0237473,MR0240082,MR0263790}.

The results discussed above do not give explicit generators. 
However, Kawamoto \cite{MR825142} gave an elementary proof of the fact that if $L/K$ is tamely ramified then 
$\mathcal{O}_{L}$ is free over $\mathcal{O}_{K}[G]$, and constructed an explicit generator along the way;
from this one easily obtains the analogous result for fractional ideals.
(Chapman \cite{MR1373957} also gave a proof of the result for fractional ideals similar to that of Kawamoto.)
The following theorem is a generalisation of these results to weakly ramified extensions.

\begin{theorem}\label{thm:main-ideal-thm}
Let $L/K$ be a weakly ramified finite Galois extension of complete local fields with finite residue fields. 
Let $G=\Gal(L/K)$ and let $n \in \Z$ such that $n \equiv 1 \bmod |G_{1}|$.
Then one can explicitly construct a free generator $\varepsilon$ of 
$\mathfrak{P}_{L}^{n}$ over $\mathcal{O}_{K}[G]$.
(The explicit description of $\varepsilon$ is given in \S \ref{sec:explicit-gen}.)
\end{theorem}

In \S \ref{sec:prelims} we cover some preliminary material, including the (well-known) constructions
of generators for unramified extensions and for totally and tamely ramified extensions.
In \S \ref{sec:splitting-lemma} we prove a `splitting lemma' that says that 
any for finite Galois extension of complete local fields $L/K$ with finite residue fields there exists a finite 
unramified extension $L'/L$ such that $L'/K$ is `doubly split' (see Definition \ref{def:splittings}).
Suppose that $L/K$ is weakly ramified.
Then $L'/K$ is also weakly ramified and we give an explicit description of a free generator $\varepsilon'$
of $\mathfrak{P}_{L'}^{n}$ over $\mathcal{O}_{K}[\Gal(L'/K)]$ in \S \ref{sec:explicit-gen};
moreover, we show that the trace $\varepsilon:=\Tr_{L'/L}(\varepsilon')$ is a
free generator of $\mathfrak{P}_{L}^{n}$ over $\mathcal{O}_{K}[G]$.
Thus we are reduced to verifying that $\varepsilon'$ is indeed a generator as claimed,
which we do as follows.
Let $p>0$ be the residue characteristic of $K$. 
In \S \ref{sec:tot-weak-p-extns} we give a short and elementary proof of the fact that
if $L/K$ is a totally and weakly ramified $p$-extension then 
any uniformizer $\pi_{L}$ is a free generator of $\mathfrak{P}_{L}$ over $\mathcal{O}_{K}[G]$;
as explained in Remark \ref{rmk:p-power-other-proofs}, this particular result has already been proven by a number of others.
In \S \ref{sec:tot-weak} we treat the case of totally and weakly ramified extensions of arbitrary degree by carefully
`glueing together' generators from two subextensions: one that is totally and weakly ramified of $p$-power degree and 
another that is totally and tamely ramified.
Finally, in \S \ref{sec:split-weak-ram} we perform
a second glueing step that crucially depends on the fact that $L'/K$ is doubly split.

We now consider the structure of $\mathcal{O}_{L}$ over its associated order $\mathfrak{A}_{L/K}$.
It is well-known that $\mathfrak{A}_{L/K}$ coincides with the group ring $\mathcal{O}_{K}[G]$ precisely when $L/K$ is tamely ramified.
Using the theory of Lubin-Tate extensions, Byott \cite[Th.\ 5]{MR1702149} showed that if 
$L/K$ is an abelian extension of $p$-adic fields that is weakly and wildly ramified, then
$\mathcal{O}_{L}$ is free over $\mathfrak{A}_{L/K}$ and, moreover,
$\mathfrak{A}_{L/K} =  \mathcal{O}_{K}[G][\pi_{K}^{-1}\Tr_{G_{0}}]$ where $\pi_{K}$ is any uniformizer of $K$ and 
$\Tr_{G_{0}} = \sum_{\tau \in G_{0}} \tau$.
Furthermore, by following the proof one can construct an explicit generator.
Byott also remarked \cite[\S 1]{MR1702149} that if $L/K$ is totally, weakly and wildly ramified 
(but not necessarily abelian) then it is straightforward to deduce the analogous statement from 
the fact that $\mathfrak{P}_{L}$ is free over $\mathcal{O}_{K}[G]$ in this case,
though one does not obtain an explicit generator in this way.
The following theorem generalises these results by having weaker hypotheses
and gives an explicit generator via Theorem \ref{thm:main-ideal-thm}; its elementary proof is given 
in \S \ref{sec:proof-of-assoc-order-result}.

\begin{theorem}\label{thm:main-valring-thm}
Let $L/K$ be a wildly and weakly ramified finite Galois extension of complete local fields with finite residue fields. 
Let $G=\Gal(L/K)$ and let $\pi_{K}$ be any uniformizer of $K$.
Then $\mathfrak{A}_{L/K} = \mathcal{O}_{K}[G][\pi_{K}^{-1}\Tr_{G_{0}}]$
and any free generator of $\mathfrak{P}_{L}$ over $\mathcal{O}_{K}[G]$
(e.g.\ as in Theorem \ref{thm:main-ideal-thm}) is also 
a free generator of $\mathcal{O}_{L}$ over $\mathfrak{A}_{L/K}$.
\end{theorem}

\subsection{Acknowledgements}
It is a pleasure to thank Alex Bartel, Nigel Byott, Griff Elder, Cornelius Greither, Derek Holt, Bernhard K\"ock and Russ Woodroofe
for helpful discussions and correspondence. 
The author is also grateful to the referee for several useful comments and suggestions.

\subsection{Conventions and Notation}
All modules are assumed to be left modules.
However, if $L/K$ is a Galois extension fields and $H \leq \Gal(L/K)$ then we let $L^{H}$ denote the subfield of $L$ fixed by $H$.
For $n \in \N$, we let $\zeta_{n}$ denote a primitive $n$th root of unity.
By a `complete local field' we mean a field that is complete with respect to a non-trivial discrete valuation; for such a field we fix the following notation:

\medskip

\begin{tabular}{ll}
$\mathcal{O}_{K}$ & the ring of integers of $K$ \\
$\mathfrak{P}_{K}$ & the maximal ideal of $\mathcal{O}_{K}$ \\
$\overline{K}$ & the residue field $\mathcal{O}_{K}/\mathfrak{P}_{K}$ \\
$\pi_{K}$ & a uniformizer of $K$ \\
$v_{K}$ & the normalised valuation $v_{K}:K^{\times} \longrightarrow \Z$
\end{tabular}

\medskip

\noindent
We make no assumptions on the residue field $\overline{K}$ except where stated otherwise.

\section{Preliminaries}\label{sec:prelims}

\subsection{A lemma on normal integral bases of ideals}
We shall make frequent use of the following easy lemma.

\begin{lemma}\label{lem:generalities-NIBs}
Let $L/K$ be a finite Galois extension of complete local fields with Galois group $G$.
Let $\mathfrak{I}$ be a non-zero fractional ideal of $\mathcal{O}_{L}$ and let 
$\overline{\mathfrak{I}} = \mathfrak{I}/\mathfrak{P}_{K}\mathfrak{I}$. 
Let $\overline{\delta} \in \overline{\mathfrak{I}}$ and let $\delta$ be any lift to $\mathfrak{I}$.
Then the following are equivalent:
\begin{enumerate}[label=\rm{(\roman*)}]
\item $\overline{\mathfrak{I}} = \overline{K}[G] \cdot \overline{\delta}$,
\item $\mathfrak{I} = \mathcal{O}_{K}[G] \cdot \delta$,
\item $\overline{\delta}$ is a free generator of $\overline{\mathfrak{I}}$ over $\overline{K}[G]$,
\item $\delta$ is a free generator of $\mathfrak{I}$ over $\mathcal{O}_{K}[G]$.
\end{enumerate}
\end{lemma}

\begin{proof}
That (i) implies (ii) is a straightforward application of Nakayama's Lemma
once one notes that $\mathfrak{P}_{K} \cdot \mathcal{O}_{K}[G]$ is a two-sided ideal contained in the Jacobson
radical of $\mathcal{O}_{K}[G]$ (see e.g.\ \cite[Prop.\ (5.22)(i)]{MR632548}); the converse is clear.
Suppose (ii) holds; then the map $\mathcal{O}_{K}[G] \longrightarrow \mathfrak{I}$ given by $x \mapsto x \cdot \delta$ is surjective
and a rank argument gives injectivity, so (iv) holds; the converse is clear.
The proof of the equivalence of (i) and (iii) is similar.
\end{proof}

\subsection{Unramified extensions}
The following result is well-known (see e.g.\ \cite[(II)]{MR825142});
we repeat the short proof for the convenience of the reader.

\begin{prop}\label{prop:unram-NIB}
Let $L/K$ be an unramified finite Galois extension of complete local fields with Galois group $G$.
Then the residue field extension $\overline{L}/\overline{K}$ is Galois and so by the Normal Basis
Theorem there exists a free generator $\overline{\beta}$ of $\overline{L}$ over $\overline{K}[G]$.
Moreover, any lift $\beta$ of such a $\overline{\beta}$ to $\mathcal{O}_{L}$ is a free generator of $\mathcal{O}_{L}$ over $\mathcal{O}_{K}[G]$.
\end{prop}

\begin{proof}
By definition of unramified, $\overline{L}/\overline{K}$ is separable (see \cite[Ch.\ I, \S 4]{MR554237}).
Thus the hypotheses imply that $\overline{L}/\overline{K}$ is in fact Galois and that
$G$ identifies with $\Gal(\overline{L}/\overline{K})$ (see  \cite[Ch.\ III, \S 5]{MR554237}).
For the last claim, apply Lemma \ref{lem:generalities-NIBs} with $\mathfrak{I}=\mathcal{O}_{L}$
noting that $\mathfrak{P}_{K}\mathcal{O}_{L} = \mathfrak{P}_{L}$.
\end{proof}

\begin{remark}\label{rmk:finite-field-normal-basis-element}
Let $L/K$ be an unramified finite extension of complete local fields with finite residue fields; then $L/K$ is necessarily Galois (in fact cyclic).
In this case, the construction of a normal basis element $\overline{\beta}$ for $\overline{L}/\overline{K}$
is a significant problem in its own right and there is a large amount of literature on the subject;
see e.g.\ \cite{MR942137}. 
We note that if $[L:K]$ is a power of $p:=\operatorname{char} \overline{K}$ then
Proposition \ref{prop:non-zero-trace} below can be applied to show that the normal basis elements of
$\overline{L}/\overline{K}$ are precisely those elements $\overline{\beta}$ 
such that $\Tr_{\overline{L}/\overline{K}}(\overline{\beta}) \neq 0$.
\end{remark}

\subsection{Totally and tamely ramified extensions}
The following lemma is well-known (see e.g.\ \cite[Ch.\ 16]{MR1885791}).

\begin{lemma}\label{lem:total-tame-Galois-Kummer}
Let $L/K$ be a totally and tamely ramified finite extension of complete local fields with finite residue fields.
Let $e=[L:K]$.
\begin{enumerate}[label=\rm{(\roman*)}]
\item  There exist uniformizers $\pi_{L}$ and $\pi_{K}$ in $L$ and $K$
respectively such that $\pi_{L}^{e}=\pi_{K}$.
\item Assume further that $L/K$ is Galois. Then $K$ contains the $e$th roots of unity and 
$L/K$ is a cyclic Kummer extension with Kummer generator $\pi_{L}$.
\end{enumerate}
\end{lemma}

The following proposition is a slight generalisation of \cite[(I)]{MR825142};  we give essentially the same proof for the convenience
of the reader. Also see \cite[\S 7.1]{MR1128708} and \cite[\S 3]{MR1373957}.

\begin{prop}\label{prop:tot-tame-gen}
Let $L/K$ be a totally and tamely ramified finite Galois extension of complete local fields with finite residue fields.
Let $e=[L:K]$ and let $G=\Gal(L/K)$.
Let $\pi_{L}$ be as in Lemma \ref{lem:total-tame-Galois-Kummer} (i) and let $\alpha \in \mathcal{O}_{L}$. 
Then there exist unique $u_{0}, \ldots, u_{e-1} \in \mathcal{O}_{K}$ such that
\[
\alpha = u_{0} + u_{1}\pi_{L} + \cdots + u_{e-1}\pi_{L}^{e-1}.
\]
Let $n \in \Z$. Then $\pi_{L}^{n}\alpha$ is a free generator of $\mathfrak{P}_{L}^{n}$ over $\mathcal{O}_{K}[G]$
if and only if $u_{i} \in \mathcal{O}_{K}^{\times}$ for $i=0,\ldots,e-1$; in particular this is the case if we take 
$\alpha= 1 + \pi_{L} + \cdots + \pi_{L}^{e-1}$.
\end{prop}

\begin{proof}
Since $L/K$ is totally ramified, we have $\mathcal{O}_{L}=\mathcal{O}_{K}[\pi_{L}]$
(see \cite[I, \S 6, Prop.\ 18]{MR554237}); this gives the first claim.
By Lemma \ref{lem:total-tame-Galois-Kummer} (ii), $L/K$ is a cyclic Kummer extension of degree $e$
with Kummer generator $\pi_{L}$. Let $\sigma$ be any generator of $G$. 
Then there exists a primitive $e$th root of unity $\zeta=\zeta_{e}$ such that $\sigma(\pi_{L})=\zeta\pi_{L}$.
Note that since $\sigma^{j}(\pi_{L}^{n}) = \zeta^{jn}\pi_{L}^{n}$ and $\zeta^{jn} \in \mathcal{O}_{K}^{\times}$
for all $j$, we are reduced to considering the case $n=0$.

A straightforward computation shows that with $A:=(u_{j} \zeta^{ij})_{0 \leq i,j \leq e-1}$ we have
\begin{equation}\label{eq:matrix-eq}
(\alpha, \sigma(\alpha), \ldots, \sigma^{e-1}(\alpha)) = (1, \pi_{L}, \ldots, \pi_{L}^{e-1})A.
\end{equation}
Since $A$ has coefficients in $\mathcal{O}_{K}$ and the vectors in \eqref{eq:matrix-eq}
give $\mathcal{O}_{K}$-bases for $\mathcal{O}_{K}[G] \cdot \alpha$ and $\mathcal{O}_{L}$, respectively,
we have that $\alpha$ is a free generator of $\mathcal{O}_{L}$ over $\mathcal{O}_{K}[G]$
if and only if $\det(A) \in \mathcal{O}_{K}^{\times}$.
Now setting $B:=(\zeta^{ij})_{0 \leq i,j \leq e-1}$ we have
\begin{equation}\label{eqn:detA}
\textstyle{\det(A) = (\prod_{k=0}^{e-1} u_{k}) \det(B)}.
\end{equation}
However, $B$ is a Vandermonde matrix, so for some $m \in \N$ we have 
\begin{equation}\label{eqn:detB}
\textstyle{\det(B) = \prod_{0 \leq i < j \leq e-1}(\zeta^{j}-\zeta^{i}) = \zeta^{m} \prod_{0 \leq i < j \leq e-1}(\zeta^{j-i}-1)}.
\end{equation}
Now consider
\[
\textstyle{f(X) := X^{e-1}+X^{e-2}+ \cdots +X + 1 = \prod_{k=1}^{e-1} (X-\zeta^{k}).}
\]
For $1 \leq k \leq e-1$, we see that $(1-\zeta^{k})$ divides $f(1)=e$.
However, since $L/K$ is tamely ramified, $e$ is relatively prime to the residue characteristic of $K$.
Hence for $1 \leq k \leq e-1$, we have $(1-\zeta^{k}) \in \mathcal{O}_{K}^{\times}$; 
thus by \eqref{eqn:detB} $\det(B) \in \mathcal{O}_{K}^{\times}$, and so by \eqref{eqn:detA}
$\det(A) \in \mathcal{O}_{K}^{\times}$ if and only if
 $u_{i} \in \mathcal{O}_{K}^{\times}$ for $i=0,\ldots,e-1$.
\end{proof}

\section{A splitting lemma for local fields}\label{sec:splitting-lemma}

\begin{definition}\label{def:splittings}
Let $L/K$ be a finite Galois extension of complete local fields with finite residue fields. 
Let $G=\Gal(L/K)$, let $I=G_{0}$ be its inertia subgroup and let $W=G_{1}$ be its wild inertia subgroup.
We say that $L/K$ is 
\begin{enumerate}[label=\rm{(\roman*)}]
\item \emph{split with respect to inertia} if $G$ decomposes as a semi-direct product $G=I \rtimes U$ for some (necessarily cyclic) 
subgroup $U$ of $G$ (so $L/L^{U}$ is unramified); 
\item \emph{split with respect to wild inertia} if $G$ decomposes as a semi-direct product $G=W \rtimes T$ for some subgroup 
$T$ of $G$ (so $L/L^{T}$ is tamely ramified);
\item \emph{doubly split} if there exists a (necessarily cyclic) subgroup $C$ of $I$
and both (i) and (ii) hold with choices of $U$ and $T$ such that there are semi-direct product decompositions 
$I=W \rtimes C$ and $T=C \rtimes U$, and so we have
\[
G=W \rtimes T= W \rtimes (C  \rtimes U) = (W \rtimes C) \rtimes U = I \rtimes U.
\]
\end{enumerate}
\end{definition}

\begin{remark}\label{rmk:totally-ram-implies-doubly-split}
If $L/K$ is totally ramified then the Schur-Zassenhaus Theorem \cite[6.2.1]{MR2014408} shows
that $L/K$ is split with respect to wild inertia and thus trivially is also doubly split.
\end{remark}

\begin{lemma}\label{lem:splitting-lemma}
Let $L/K$ be a finite Galois extension of complete local fields with finite residue fields.
Let $d$ be any positive integer divisible by the exponent of $\Gal(L/K)$ (e.g.\ take $d=[L:K]$).
Let $K'/K$ be the unique unramified extension of degree $d$ and let $L' = LK'$. 
Then $L'/K$ is Galois, $\Gal(L'/K')$ is the inertia subgroup of $\Gal(L'/K)$,
and $L'/K$ is doubly split.
\end{lemma}

\begin{proof}
Since $L/K$ and $K'/K$ are both Galois, so is $L'/K$.
By considering ramification degrees, it is straightforward to check that $I:=\Gal(L'/K')$ is
the inertia subgroup of $G:=\Gal(L'/K)$.

We show that $L'/K$ is split with respect to inertia. 
Consider the exact sequence
\begin{equation}\label{eq:split-exact-galois}
1 \longrightarrow I=\Gal(L'/K') \longrightarrow G=\Gal(L'/K) \stackrel{\rho}{\longrightarrow} 
\Gal(K'/K) \longrightarrow 1.
\end{equation}
Let $\sigma \in \Gal(K'/K)$ be the Frobenius element (or indeed any generator of this cyclic group)
and take any $\tau \in \Gal(L'/K)$ with $\rho(\tau)=\sigma$. 
Then $\tau^{d}$ is the identity on both $L$ and $K'$, 
so we have $\tau^{d}=\mathrm{id}_{L'}$.
Therefore $\varphi:\Gal(K'/K) \longrightarrow \Gal(L'/K)$, defined by $\varphi(\sigma)=\tau$,
is a splitting homomorphism for \eqref{eq:split-exact-galois}.
Thus we may take $U=\langle \tau \rangle$.

We now prove that $L'/K$ is in fact doubly split.
Let $p>0$ be the residue characteristic of $K$ and let 
 $W$ (wild inertia) be the unique Sylow $p$-subgroup of $I$.
Since $|I/W|$ is coprime to $p$, by the first claim of Schur-Zassenhaus Theorem \cite[6.2.1]{MR2014408} 
there exists a (cyclic) complement $C$ of $W$ in $I$ (i.e.\ $I=WC$ and $W \cap C = 1$).
Let $N=N_{G}(C)$ be the normaliser of $C$ in $G$.
Since $C$ is soluble, the second claim of the Schur-Zassenhaus Theorem \cite[6.2.1]{MR2014408} says that all complements
of $W$ in $I$ are conjugate (to $C$), and so the Frattini argument \cite[3.1.4]{MR2014408} shows that $G=IN$. 
Hence we can and do assume that $\tau$ defined in the above paragraph in fact belongs to $N$.
Recall that $U =\langle \tau \rangle$ is a complement of $I$ in $G$.
Moreover, $U \leq N=N_{G}(C)$ and so $T:=CU$ is a subgroup of $G$.
Note that $I \cap U = 1$ and $C \leq I$, so $C \cap U=1$.
Thus $U$ is a complement of $C$ in $T$ and we have $|T|=|C|\cdot|U|$.
Now we have 
$G = IU = (WC)U = W(CU) = WT$.
Moreover, $|G|=|W|\cdot|T|$ and so $W \cap T=1$. 
Therefore $T$ is the desired complement of $W$ in $G$.
\end{proof}

\begin{remark}\label{rmk:ideas-for-proofs}
The second paragraph of the proof of Lemma \ref{lem:splitting-lemma} is an adaptation of the proof of \cite[Lem.\ 1]{MR1627831}, 
which shows that $L'/K$ is split with respect to inertia when $L/K$ is abelian.
The author is grateful to both Derek Holt for a helpful discussion that led to the argument used in final paragraph
of the proof of Lemma \ref{lem:splitting-lemma}, and to Russ Woodroofe for pointing out that 
Gasch\"utz's Theorem \cite[3.3.2]{MR2014408} can be used to give an alternative proof of this result in a special case
(see the MathOverflow discussion \cite{156209}).
\end{remark}

\section{The explicit description of a generator}\label{sec:explicit-gen}

\begin{theorem}\label{thm:explicit-gen}
Let $L/K$ be a weakly ramified finite Galois extension of complete local fields with finite residue fields. 
Let $G=\Gal(L/K)$ and let $n \in \Z$ such that $n \equiv 1 \bmod |G_{1}|$.
Suppose that $L/K$ is doubly split in the sense of Definition \ref{def:splittings}
and let $I,W,T,U,C$ have the meanings given therein.
\begin{itemize}
\item Let $p>0$ be the residue characteristic of $K$.
\item Define $r$ by $p^{r}=|G_{1}|=|W|$ and let $c=|C|$.
\item Let $a,b \in \Z$ such that $ap^{r}+bc=1$ (note that $p \nmid c$).
\item Let $\pi_{T}$ be any uniformizer of $L^{T}$.
\item Let $S=WU$ (this is a subgroup of $G$ since $W$ is normal in $G$).
\item Let $\pi_{S}$ be a uniformizer of $L^{S}$ such that $\pi_{S}^{c}$ is a uniformizer of $K$\\
(since $L^{S}/K$ is totally and tamely ramified, this is possible by Lemma \ref{lem:total-tame-Galois-Kummer}).
\item For $i=0,\ldots,c-1$ let $u_{i} \in \mathcal{O}_{K}^{\times}$ (e.g.\ take $u_{0}= \cdots =u_{c-1}=1$).
\item Let $\alpha = u_{0} + u_{1}\pi_{S} + u_{2}\pi_{S}^{2} +\cdots + u_{c-1}\pi_{S}^{c-1}$. 
\item Let $\beta$ be a normal integral basis generator for the unramified extension $L^{I}/K$
(such an element exists by Proposition \ref{prop:unram-NIB}.)
\end{itemize}
Then $\pi_{T}^{nb} \pi_{S}^{na} \alpha\beta$ is a free generator of $\mathfrak{P}_{L}^{n}$ over $\mathcal{O}_{K}[G]$.
\end{theorem}

\begin{proof}
This is proven in \S \ref{sec:split-weak-ram} and builds on the proof for totally ramified extensions
given in \S \ref{sec:tot-weak}, which in turn uses the result for totally ramified $p$-extensions 
proven in \S \ref{sec:tot-weak-p-extns}.
\end{proof}

\begin{theorem}\label{thm:trace-explicit-gen}
Let $L/K$ be a weakly ramified finite Galois extension of complete local fields with finite residue fields.
Let $G=\Gal(L/K)$ and let $n \in \Z$ such that $n \equiv 1 \bmod |G_{1}|$.
Let $d$ be any positive integer divisible by the exponent of $G$ (e.g.\ take $d=[L:K]$).
Let $K'/K$ be the unique unramified extension of degree $d$ and let $L'=LK'$.
Then $L'/K$ is Galois, weakly ramified, and doubly split in the sense of 
Definition \ref{def:splittings}. 
Let $\varepsilon' \in L'$ be any free generator of $\mathfrak{P}_{L'}^{n}$ over $\mathcal{O}_{K}[\Gal(L'/K)]$
(e.g.\ as in Theorem \ref{thm:explicit-gen}). 
Then $\varepsilon:=\Tr_{L'/L}(\varepsilon')$ is a free generator of $\mathfrak{P}_{L}^{n}$ over $\mathcal{O}_{K}[G]$.
\end{theorem}

\begin{proof}
Lemma \ref{lem:splitting-lemma} shows that $L'/K$ is Galois and doubly split.
Since $L'/L$ is unramified, \cite[Prop.\ 4.4]{MR1702149} implies that $L'/K$ is weakly ramified
and \cite[III, \S 3, Prop.\ 7]{MR554237} shows that
$\Tr_{L'/L}(\mathfrak{P}_{L'}^{n})=\mathfrak{P}_{L}^{n}$. 
Thus we obtain
\begin{align*}
\mathfrak{P}_{L}^{n} 
&= \Tr_{L'/L}(\mathcal{O}_{K}[\Gal(L'/K)] \cdot \varepsilon' ) \\
&= \mathcal{O}_{K}[\Gal(L'/K)] \cdot \Tr_{L'/L}(\varepsilon') \\
&= \mathcal{O}_{K}[G] \cdot \Tr_{L'/L}(\varepsilon').
\end{align*}
Applying Lemma \ref{lem:generalities-NIBs} with $\mathfrak{I}=\mathfrak{P}_{L}^{n}$ now 
gives the desired result.
\end{proof}

\begin{remark}
Theorem \ref{thm:main-ideal-thm} follows from Theorems \ref{thm:explicit-gen} and \ref{thm:trace-explicit-gen}. 
When specialised to the tamely ramified case, the proof essentially reduces to the proof of
Kawamoto \cite{MR825142}, and we recover the main result given therein.
\end{remark}

\section{Totally and weakly ramified $p$-extensions}\label{sec:tot-weak-p-extns}

We start by giving a slight generalisation of part of \cite[Th.\ 1]{MR606253}
(also see \cite[\S 18, Ex.\ 3]{MR632548} or \cite[Prop.\ 7]{MR2457723}).

\begin{prop}\label{prop:non-zero-trace}
Let $p$ be prime, let $k$ be any field of characteristic $p$ and let $G$ be any finite $p$-group.
Let $M$ be a left $k[G]$-module such that $\dim_{k} M = |G|$ and let $\Tr_{G} = \sum_{g \in G} g$.
Let $x \in M$. Then $x$ is a free generator of $M$ over $k[G]$ if and only if $\Tr_{G} \cdot x \neq 0$.
\end{prop}

\begin{proof}
Let $m_{x}: k[G] \longrightarrow M$ be the $k[G]$-homomorphism given by $y \mapsto y \cdot x$.
In particular, $m_{x}$ is a $k$-linear map with domain and codomain of equal finite dimension.
Hence $m_{x}$ is a bijection if and only if $\Ann_{k[G]}(x)$ is trivial. 
However, by \cite[Cor.\ (a)]{MR606253} (or \cite[\S 18, Ex.\ 2]{MR632548} or \cite[Prop.\ 6]{MR2457723})
the group algebra $k[G]$ has a unique minimal (left) ideal $k[G] \cdot \Tr_{G} = k \cdot \Tr_{G}$.
Thus $\Ann_{k[G]}(x)$ is trivial if and only if $\Tr_{G} \notin \Ann_{k[G]}(x)$.
\end{proof}

\begin{theorem}\label{thm:tot-weak-p-extension}
Let $K$ be a complete local field with perfect residue field of characteristic $p>0$.
Let $L/K$ be a totally and weakly ramified finite Galois $p$-extension and let $n \in \Z$.
\begin{enumerate}[label=\rm{(\roman*)}]
\item The Galois group $G:=\Gal(L/K)$ is an elementary abelian $p$-group.
\item The ideal $\mathfrak{P}_{L}^{n}$ is a free (rank $1$) $\mathcal{O}_{K}[G]$-module if and only if 
$n \equiv 1 \bmod |G|$.
\item Suppose $n \equiv 1 \bmod |G|$. 
Then $\delta \in L$ is a free generator of $\mathfrak{P}_{L}^{n}$ over $\mathcal{O}_{K}[G]$ if and only if
$v_{L}(\delta)=n$.
\end{enumerate}
\end{theorem}

\begin{remark}\label{rmk:p-power-other-proofs}
Using the theory of Galois scaffolds, 
Theorem \ref{thm:tot-weak-p-extension} (iii) is also proven in recent work of Byott and Elder \cite[Prop.\ 4.4]{large_galois_scaffolds}
when $n=1$ (in fact, op.\ cit.\ also gives the analogous result for $\mathcal{O}_{L}$ over its associated order $\mathfrak{A}_{L/K}$); 
the result for general $n \equiv 1 \bmod |G|$ is trivial to deduce from this.
Moreover, there are two other proofs of Theorem \ref{thm:tot-weak-p-extension} 
(iii) in the literature in the case that  $n=1$ and $K$ is a $p$-adic field:
Vostokov \cite[Prop.\ 2]{0462.12004} proved the result by a direct computation;
using the theory of Lubin-Tate extensions, Byott \cite[Cor.\ 4.3]{MR1702149} showed that any uniformizer $\pi_{L}$ of $L$
is a free generator of $\mathcal{O}_{L}$ over $\mathcal{O}_{K}[G][\pi_{K}^{-1}\Tr_{G}]$,
and from this Vinatier \cite[Prop.\ 2.4]{MR2167720} deduced the result.
One advantage of the proof below is that it is short, elementary and largely self-contained.
\end{remark}

\begin{example}
Let $K$ be a finite unramified extension of $\Q_{p}$ and let $L$ be the unique intermediate field of the extension 
$K(\zeta_{p^{2}})/K$ such that $[L:K]=p$. 
Then it is straightforward to check that $L/K$ is a totally and weakly ramified extension.
Thus any uniformizer $\pi_{L}$ of $L$ is a free generator of $\mathfrak{P}_{L}$ over $\mathcal{O}_{K}[\Gal(L/K)]$.
\end{example}

\begin{example}
Let $K=\F((t))$ be a local function field with perfect residue field $\F$ of characteristic $p>0$.
Let $L=K(x)$ where $x$ satisfies $x^{p}-x = \pi_{K}^{-1}$ where $\pi_{K}$ is any uniformizer of $K$
(e.g.\ $\pi_{K}=t$). Then by Artin-Schreier theory, $L/K$ is a cyclic Galois extension of degree $p$.
It is straightforward to check that $L/K$ is totally and weakly ramified, and so 
any uniformizer $\pi_{L}$ of $L$ (e.g.\ $\pi_{L}=x^{-1}$) is a free generator of $\mathfrak{P}_{L}$ over $\mathcal{O}_{K}[\Gal(L/K)]$.
\end{example}

\begin{proof}[Proof of Theorem \ref{thm:tot-weak-p-extension}]
Part (i) is standard and follows from the hypotheses and
the fact that $G_{1}/G_{2}$ is always an elementary 
abelian $p$-group (see \cite[IV, \S 2, Cor.\ 3]{MR554237}).

Let $\mathfrak{D}_{L/K}$ denote the different of $L/K$. 
Then as $L/K$ is weakly ramified, 
Hilbert's formula (\cite[IV, \S 1, Prop.\ 4]{MR554237}) shows that 
$v_{L}(\mathfrak{D}_{L/K})=2|G|-2$.
Now from \cite[III, \S 3, Prop.\ 7]{MR554237} 
it follows that for any $i \in \Z$ we have
\begin{equation}\label{eq:trace-formula}
\Tr_{G}(\mathfrak{P}_{L}^{i}) 
= \Tr_{L/K}(\mathfrak{P}_{L}^{i}) 
= \mathfrak{P}_{K}^{2 + \left\lfloor \frac{i-2}{|G|} \right\rfloor}
\end{equation}
where $\lfloor x \rfloor$ denotes the largest $k \in \Z$ such that $k \leq x$.
For $i \in \Z$ define
\[
\overline{\mathfrak{P}_{L}^{i}} := \mathfrak{P}_{L}^{i} / \mathfrak{P}_{K}\mathfrak{P}_{L}^{i}
= \mathfrak{P}_{L}^{i} / \mathfrak{P}_{L}^{|G|+i}.
\]
Then by \eqref{eq:trace-formula} we have
\[
\Tr_{G}(\overline{\mathfrak{P}_{L}^{i}}) 
= \frac{\Tr_{G}(\mathfrak{P}_{L}^{i}) + \mathfrak{P}_{L}^{|G|+i}}{\mathfrak{P}_{L}^{|G|+i}}
= \left\{ 
\begin{array}{cl}
\mathfrak{P}_{L}^{|G|+i-1}/\mathfrak{P}_{L}^{|G|+i} & \textrm{if } i \equiv 1 \bmod |G|,  \\
0 & \textrm{otherwise}.
\end{array}
\right.
\]
Hence if $n \not \equiv 1 \bmod |G|$, by Proposition \ref{prop:non-zero-trace} we have that 
$\overline{\mathfrak{P}_{L}^{n}} \neq \overline{K}[G] \cdot \overline{\delta}$ for all 
$\overline{\delta} \in \overline{\mathfrak{P}_{L}^{n}}$,
and so $\mathfrak{P}_{L}^{n}$ is not free over $\mathcal{O}_{K}[G]$ by 
Lemma \ref{lem:generalities-NIBs} with $\mathfrak{I} = \mathfrak{P}_{L}^{n}$.

Now suppose $n \equiv 1 \bmod |G|$. 
Let $\theta: \overline{\mathfrak{P}_{L}^{n}} \longrightarrow \overline{\mathfrak{P}_{L}^{n}}$
be defined by $x \mapsto \Tr_{G} \cdot x$. 
Then $\theta$ is a $\overline{K}$-linear map with $\dim_{\overline{K}} \operatorname{im} \theta = 1$
and so $\dim_{\overline{K}} \ker \theta = |G|-1$.
Furthermore, $\mathfrak{K} := \mathfrak{P}_{L}^{n+1}/\mathfrak{P}_{L}^{n+|G|}$ is a 
$\overline{K}[G]$-submodule of $\overline{\mathfrak{P}_{L}^{n}}$ 
with $\dim_{\overline{K}} \mathfrak{K} = |G|-1$ and by \eqref{eq:trace-formula} we have
\[
\Tr_{G}(\mathfrak{K}) 
= \frac{\Tr_{G}(\mathfrak{P}_{L}^{n+1}) + \mathfrak{P}_{L}^{n+|G|}}{ \mathfrak{P}_{L}^{n+|G|}}
=  \frac{\mathfrak{P}_{K}^{2 + \left\lfloor \frac{n-1}{|G|} \right\rfloor } + \mathfrak{P}_{L}^{n+|G|}}{ \mathfrak{P}_{L}^{n+|G|}}  
= 0.
\]
Hence $\mathfrak{K} \leq \ker \theta$ and this containment is in fact an equality as  
both spaces are of equal finite dimension over $\overline{K}$.
Thus by Proposition \ref{prop:non-zero-trace} we have that
\[
\overline{K}[G] \cdot \overline{\delta} = \overline{\mathfrak{P}_{L}^{n}} \Longleftrightarrow \overline{\delta} \in \overline{\mathfrak{P}_{L}^{n}} - \mathfrak{K} = 
\mathfrak{P}_{L}^{n} / \mathfrak{P}_{L}^{|G|+n} - \mathfrak{P}_{L}^{n+1} / \mathfrak{P}_{L}^{|G|+n}.
\]
Therefore by Lemma \ref{lem:generalities-NIBs} with $\mathfrak{I}=\mathfrak{P}_{L}^{n}$
we see that $\delta \in L$ is a free generator of $\mathfrak{P}_{L}^{n}$ over
$\mathcal{O}_{K}[G]$ if and only if $v_{L}(\delta)=n$. 
\end{proof}

\section{Totally and weakly ramified extensions of arbitrary degree}\label{sec:tot-weak}

Let $M$ be a complete local field with finite residue field of characteristic $p$.
Let $L/M$ be a totally and weakly ramified finite Galois extension and let $I=G_{0}=\Gal(L/M)$.
Since $G_{2}$ is trivial, $W:=G_{1}$ is an elementary abelian $p$-group. 
By Remark \ref{rmk:totally-ram-implies-doubly-split},
$L/M$ is split with respect to wild inertia, i.e., 
$I$ decomposes as a semi-direct product $I = W \rtimes C$ 
for some cyclic subgroup
$C$ of $I$. (Note that as $L/M$ is totally ramified, we can and do write $C$ instead of $T$ here; this is consistent with the notation used in \S
\ref{sec:split-weak-ram}.)
Let $E=L^{W}$ and $F=L^{C}$ be the subfields of $L$ fixed by $W$ and $C$, respectively.
Note that the choice of $C$ (and hence of $F$) 
is not necessarily unique and that the order of $C$ is prime to $p$.
We identify $\Gal(E/M)$ with $C=\Gal(L/F)$ via the restriction map 
$C \rightarrow \Gal(E/M)$, $\gamma \mapsto \gamma|_{E}$.
The situation is represented by the following field diagram.
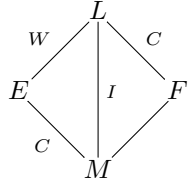
\[
\xymatrix@1@!0@=30pt { 
& L &  \\
E \ar@{-}[ur]^{W} & & F \ar@{-}[ul]_{C}\\
& M  \ar@{-}[uu]_{I} \ar@{-}[ur] \ar@{-}[ul]^{C} & \\
}
\]
Both $L/E$ and $F/M$ are are totally and wildly ramified $p$-extensions and both $L/F$ and $E/M$ are totally and tamely ramified. Note that $F/M$ need not be Galois.

Define $r$ by $p^{r}=[L:E]=[F:M]=|W|$ and let $c=[L:F]=[E:M]=|C|$. 
Since $E/M$ is totally and tamely ramified, by Lemma \ref{lem:total-tame-Galois-Kummer} (i)
there exist uniformizers $\pi_{E}$ and $\pi_{M}$
of $E$ and $M$ respectively such that $\pi_{E}^{c} = \pi_{M}$.
By B\'ezout's Lemma, there exist integers $a,b$ such that $ap^{r} + bc=1$.

\begin{prop}\label{prop:tot-weak}
Let $n \in \Z$ such that $n \equiv 1 \bmod |W|$.
For $i=0,\ldots,c-1$ let $u_{i} \in \mathcal{O}_{M}^{\times}$.
Let $\pi_{E}$ be a uniformizer chosen as above, let $\alpha = u_{0} + u_{1}\pi_{E} + \cdots + u_{c-1}\pi_{E}^{c-1}$,
and let $\pi_{F}$ be any uniformizer of $F$.
Then $\pi_{F}^{nb} \pi_{E}^{na} \alpha$ is a free generator of $\mathfrak{P}_{L}^{n}$ over $\mathcal{O}_{M}[I]$.
\end{prop}

\begin{proof}
Write $W  = \{ \tau_{i} \}$ and $C=\{ \sigma_{j} \}$.
Since $L/M$ is weakly ramified, it follows directly from the definition of the ramification groups that 
$L/E$ is also weakly ramified. 
Hence by Theorem \ref{thm:tot-weak-p-extension} (iii) any $\delta \in L$ with $v_{L}(\delta)=n$
is a free generator of $\mathfrak{P}_{L}^{n}$ over $\mathcal{O}_{E}[W]$. 
However, we have
$v_{L}(\pi_{F}^{b}\pi_{E}^{a}) = bc + ap^{r} = 1$, and so in particular we may 
take $\delta= \pi_{F}^{nb}\pi_{E}^{na}$.
Furthermore, by Proposition \ref{prop:tot-tame-gen} we have that 
$\pi_{E}^{na}\alpha$ is a free generator of the fractional ideal $\pi_{E}^{na}\mathcal{O}_{E}$ over 
$\mathcal{O}_{M}[C]$. 
Therefore we have 
\begin{align*}
\mathfrak{P}_{L}^{n} 
&= \mathcal{O}_{E}[W] \cdot (\pi_{F}^{nb}\pi_{E}^{na}) \\
&= \bigoplus_{i} \tau_{i}(\pi_{F}^{nb}\pi_{E}^{na})\mathcal{O}_{E} \\
&=  \bigoplus_{i} \tau_{i}(\pi_{F}^{nb})(\pi_{E}^{na}\mathcal{O}_{E})
\quad \textrm{since } \pi_{E} \in E=L^{W} \\
&= \bigoplus_{i} \tau_{i}(\pi_{F}^{nb}) (\mathcal{O}_{M}[C] \cdot \pi_{E}^{na}\alpha)\\
&= \bigoplus_{i} \tau_{i}(\pi_{F}^{nb}) \bigoplus_{j} \sigma_{j}(\pi_{E}^{na}\alpha) \mathcal{O}_{M}\\
&= \bigoplus_{i}  \bigoplus_{j} \tau_{i}(\pi_{F}^{nb}) \sigma_{j}(\pi_{E}^{na}\alpha) \mathcal{O}_{M}\ \\
&= \bigoplus_{i}  \bigoplus_{j} \tau_{i} \sigma_{j}(\pi_{F}^{nb}) \sigma_{j}(\pi_{E}^{na}\alpha) \mathcal{O}_{M} 
\quad \textrm{since } \pi_{F} \in F=L^{C}\\
&= \bigoplus_{i}  \bigoplus_{j} \tau_{i} \sigma_{j}(\pi_{F}^{nb}) \tau_{i}\sigma_{j}(\pi_{E}^{na}\alpha) \mathcal{O}_{M} 
\quad \textrm{since } \sigma_{j}(\pi_{E}^{na}\alpha) \in E=L^{W}\\
&= \bigoplus_{i}  \bigoplus_{j} \tau_{i} \sigma_{j}(\pi_{F}^{nb}\pi_{E}^{na}\alpha) \mathcal{O}_{M}\\
&= \mathcal{O}_{M}[I] \cdot (\pi_{F}^{nb}\pi_{E}^{na}\alpha).
\end{align*}
The result now follows from Lemma \ref{lem:generalities-NIBs} with $\mathfrak{I}=\mathfrak{P}_{L}^{n}$.
\end{proof}

\begin{remark}
The author is grateful to Nigel Byott for the following observation and to the referee for suggestions regarding explicit examples. 
If $L/K$ is abelian, not of $p$-power degree, and totally and wildly ramified, then $L/K$ cannot be weakly ramified 
(see e.g.\ \cite[IV, \S 2, Cor.\ 2]{MR554237}).
In particular, $\Q_{p}(\zeta_{p^{2}})/\Q_{p}$ is not weakly ramified, even though the subextension 
$\Q_{p}(\zeta_{p^{2}})/\Q_{p}(\zeta_{p})$ is weakly ramified.
However, there do exist non-abelian Galois extensions of local fields, not of $p$-power degree, that are totally, wildly and weakly ramified.
For example, let $K=\Q_{3}$ and let $L=\Q_{3}(\zeta_{3}, \sqrt[3]{2})$. 
Then $L/K$ is Galois with $\Gal(L/K) \simeq S_{3}$, the symmetric group on three letters.
Furthermore, $L/K$ is totally, wildly and weakly ramified. 
\end{remark}

\section{Weakly ramified extensions that are doubly split}\label{sec:split-weak-ram}

Let $K$ be a complete local field with finite residue field of characteristic $p$.
Let $L/K$ be a weakly ramified finite Galois extension and let $G=\Gal(L/K)$.
Suppose that $L/K$ is doubly split and adopt the notation of Definition \ref{def:splittings}.
Let $M=L^{I}$ be the inertia subfield and let $N=L^{U}$. 
Note that the choice of $U$ (and hence of $N$) 
is not necessarily unique.
We identify $\Gal(M/K)$ with $U=\Gal(L/N)$ via the restriction map 
$U \rightarrow \Gal(M/K)$, $\gamma \mapsto \gamma|_{M}$.
The extension $L/M$ `decomposes' exactly as in \S \ref{sec:tot-weak}
and we henceforth assume all the notation used therein.
The situation is represented by the following pair of field diagrams.
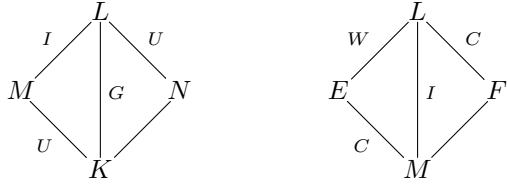
\[
\xymatrix@1@!0@=30pt { 
& L &  & &  & L & \\
M \ar@{-}[ur]^{I} &  & N \ar@{-}[ul]_{U} & & E \ar@{-}[ur]^{W} & & F \ar@{-}[ul]_{C}\\
& K   \ar@{-}[uu]_{G} \ar@{-}[ur] \ar@{-}[ul]^{U} & & & & M   \ar@{-}[uu]_{I} \ar@{-}[ur] \ar@{-}[ul]^{C} &\\
}
\]

We note that $S:=WU$ is a subgroup of $G$ since $W$ is normal in $G$ and that $T=CU$ is a subgroup of $G$ by hypothesis.
Thus 
\[
E \cap N = L^{W} \cap L^{U}= L^{WU}=L^{S} \quad \textrm{and}
\quad F \cap N = L^{C} \cap L^{U} = L^{CU} = L^{T}.
\]
Furthermore, $W$, $I$ and $C$ are normal in $S$, $G$ and $T$, respectively. 
Therefore we have the following field diagram in which we have identified $U$ with the Galois group of the relevant extensions 
via restriction maps as above, and unmarked extensions are not necessarily Galois.
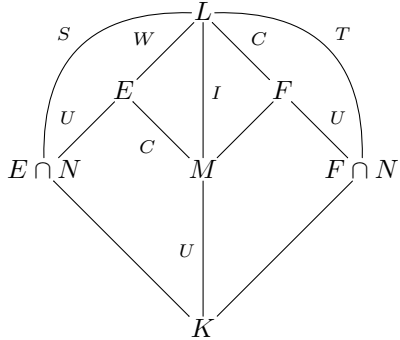
\[
\xymatrix@1@!0@=30pt { 
& & L \ar@/^2pc/@{-}[ddrr]^{T}  & &  \\
& E \ar@{-}[ur]^{W} \ar@{-}[dl]_{U} & & F \ar@{-}[dr]^{U} \ar@{-}[ul]_{C} & \\
E \cap N  \ar@/^2pc/@{-}[uurr]^{S} \ar@{-}[ddrr] & & M  \ar@{-}[uu]_{I} \ar@{-}[dd]_{U} \ar@{-}[ur] \ar@{-}[ul]^{C} & &  F \cap N \ar@{-}[ddll] \\
& & & & \\
& & K & &
}
\]

We now choose elements in the various intermediate fields, from which we will construct a 
free generator of $\mathfrak{P}_{L}^{n}$ over $\mathcal{O}_{K}[G]$ when $n \equiv 1 \bmod |W|$.
We adopt the notation of \S \ref{sec:tot-weak}, so that $p^{r}=[L:E]=[F:M]=|W|$, $c=[L:F]=[E:M]=|C|$,
and  $a,b \in \Z$ satisfy $ap^{r}+bc=1$.
Let $\pi_{T}$ be any uniformizer of $L^{T}=F \cap N$. 
Let $\pi_{S}$ be a uniformizer of $L^{S}=E \cap N$ such that $\pi_{S}^{c}$ is a uniformizer of $K$;
this is possible by Lemma \ref{lem:total-tame-Galois-Kummer} (i) since $L^{S}/K$ is totally and tamely ramified. 
Note that both $\pi_{S}$ and $\pi_{T}$ belong to $N$.
Since $M/K$ is unramified, by Proposition \ref{prop:unram-NIB} 
there exists $\beta \in \mathcal{O}_{M}$ such that
$\mathcal{O}_{M} = \mathcal{O}_{K}[U] \cdot \beta$.

\begin{prop}\label{prop:split-weak-ram}
Let $n \in \Z$ such that $n \equiv 1 \bmod |W|$.
For $i=0,\ldots,c-1$ let $u_{i} \in \mathcal{O}_{K}^{\times}$.
Let $\pi_{S}$ and $\pi_{T}$ be uniformizers chosen as above and let $\alpha = u_{0} + u_{1}\pi_{S} + \cdots + u_{c-1}\pi_{S}^{c-1}$.
Then $\pi_{T}^{nb} \pi_{S}^{na} \alpha\beta$ is a free generator of $\mathfrak{P}_{L}^{n}$ over $\mathcal{O}_{K}[G]$.
\end{prop}

\begin{proof}
Let $\gamma = \pi_{T}^{nb} \pi_{S}^{na} \alpha$. 
Note that as $E/E\cap N$ is unramified, $\pi_{S}$ is a uniformizer of $E$.
Similarly, $\pi_{S}^{c}$ is a uniformizer of $M$ and $\pi_{T}$ is a uniformizer of $F$. 
Thus by Proposition \ref{prop:tot-weak} we have $\mathfrak{P}_{L}^{n} = \mathcal{O}_{M}[I] \cdot \gamma$.
A key point is that $\gamma$ belongs to $N$ since both $\pi_{S}$ and $\pi_{T}$ were chosen to be in $N$. 
Write $I = \{ \tau_{i} \}$ and $U=\{ \sigma_{j} \}$. Then
\begin{align*}
\mathfrak{P}_{L}^{n} &= \mathcal{O}_{M}[I] \cdot \gamma\\ 
&= \bigoplus_{i} \tau_{i}(\gamma)\mathcal{O}_{M}\\
&= \bigoplus_{i} \tau_{i}(\gamma)(\mathcal{O}_{K}[U] \cdot \beta)\\
&= \bigoplus_{i} \tau_{i}(\gamma) \bigoplus_{j} \sigma_{j}(\beta) \mathcal{O}_{K}\\
&= \bigoplus_{i}  \bigoplus_{j} \tau_{i}(\gamma)\sigma_{j}(\beta) \mathcal{O}_{K} \\
&= \bigoplus_{i}  \bigoplus_{j} \tau_{i} \sigma_{j}(\gamma) \sigma_{j}(\beta) \mathcal{O}_{K} 
\quad \textrm{since } \gamma \in N=L^{U}\\
&= \bigoplus_{i}  \bigoplus_{j} \tau_{i} \sigma_{j}(\gamma) \tau_{i}\sigma_{j}(\beta) \mathcal{O}_{K} 
\quad \textrm{since } \sigma_{j}(\beta) \in M=L^{I}\\
&= \bigoplus_{i}  \bigoplus_{j} \tau_{i} \sigma_{j}(\gamma\beta) \mathcal{O}_{K}\\
&= \mathcal{O}_{K}[G] \cdot (\gamma\beta).
\end{align*}
The result now follows from Lemma \ref{lem:generalities-NIBs} with $\mathfrak{I}=\mathfrak{P}_{L}^{n}$.
\end{proof}

\section{Proof of Theorem \ref{thm:main-valring-thm}}\label{sec:proof-of-assoc-order-result}

\begin{proof}[Proof of Theorem \ref{thm:main-valring-thm}]
Let $F=L^{G_{0}}$ be the inertia subfield of $L$.
Since $L/F$ is wildly ramified, we have
$\Tr_{G_{0}}(\mathcal{O}_{L}) = \Tr_{L/F}(\mathcal{O}_{L}) \subseteq \mathfrak{P}_{F}$
(see e.g.\ \cite[Th.\ 26(b)]{MR1215934}). 
Since $F/K$ is unramified, we hence have
$\pi_{K}^{-1}\Tr_{G_{0}}(\mathcal{O}_{L}) \subseteq \pi_{K}^{-1}\mathfrak{P}_{F}=\mathcal{O}_{F} 
\subseteq \mathcal{O}_{L}$.
Therefore
\[
\mathcal{O}_{K}[G][\pi_{K}^{-1}\Tr_{G_{0}}] \subseteq \mathfrak{A}_{L/K}.
\]
Let $\varepsilon$ be a free generator of $\mathfrak{P}_{L}$ over $\mathcal{O}_{K}[G]$
(e.g.\ as in Theorem \ref{thm:main-ideal-thm}).
Then
\begin{equation}\label{eq:containment}
\mathcal{O}_{K}[G][\pi_{K}^{-1}\Tr_{G_{0}}] \cdot \varepsilon 
\subseteq \mathfrak{A}_{L/K} \cdot \varepsilon \subseteq \mathcal{O}_{L}.
\end{equation}

Let $p>0$ be the residue characteristic of $K$.
Let $S \subseteq G$ be a set of representatives of the quotient group $G/G_{0}$
and let $T=\{ \pi_{K}^{-1}s\Tr_{G_{0}} \}_{s \in S}$.
Since $p$ divides $|G_{0}|$ and $G_{0}$ is normal in $G$, 
the element $\pi_{K}^{-1}\Tr_{G_{0}}$ is an $\mathcal{O}_{K}$-multiple
of either a central idempotent (if $\operatorname{char} K = 0$) or a central nilpotent element 
(if $\operatorname{char} K = p$). 
Thus $T \cup G \cup \{ 0 \}$ is multiplicatively closed and so
$T \cup G$ is an $\mathcal{O}_{K}$-spanning set for 
$\mathcal{O}_{K}[G][\pi_{K}^{-1}\Tr_{G_{0}}]$.
Furthermore, $T$ is an $\mathcal{O}_{K}$-linearly independent set.
Therefore considering generalised module indices (see e.g.\ \cite[II.4]{MR1215934}), 
we have 
$\mathfrak{P}_{K}^{|S|} = [\mathcal{O}_{K}[G][\pi_{K}^{-1}\Tr_{G_{0}}] : \mathcal{O}_{K}[G]]_{\mathcal{O}_{K}}$.

Let $\theta : K[G] \longrightarrow L$ be the $K[G]$-module 
homomorphism given by $x \mapsto x \cdot \varepsilon$.
By definition of $\varepsilon$, the restriction of $\theta$ to $\mathcal{O}_{K}[G]$ is injective;
by extension of scalars the same is true of $\theta$ itself and of thus any restriction of $\theta$.
In particular, for any two $\mathcal{O}_{K}$-lattices $M,N$ in $K[G]$, we see that
$[M:N]_{\mathcal{O}_{K}} = [M \cdot \varepsilon:N \cdot \varepsilon ]_{\mathcal{O}_{K}}$.
Therefore
\begin{align*}
[\mathcal{O}_{L}:\mathfrak{P}_{L}]_{\mathcal{O}_{K}}
&= [\mathcal{O}_{F}:\mathfrak{P}_{F}]_{\mathcal{O}_{K}}\\
&= \mathfrak{P}_{K}^{[F:K]}\\
&= \mathfrak{P}_{K}^{|S|} \\
&= [\mathcal{O}_{K}[G][\pi_{K}^{-1}\Tr_{G_{0}}] : \mathcal{O}_{K}[G]]_{\mathcal{O}_{K}} \\
&=  [\mathcal{O}_{K}[G][\pi_{K}^{-1}\Tr_{G_{0}}] \cdot \varepsilon: \mathcal{O}_{K}[G] \cdot \varepsilon]_{\mathcal{O}_{K}} \\
&=  [\mathcal{O}_{K}[G][\pi_{K}^{-1}\Tr_{G_{0}}] \cdot \varepsilon: \mathfrak{P}_{L}]_{\mathcal{O}_{K}}. 
\end{align*}
This shows that the containments of \eqref{eq:containment} are in fact equalities.
Hence the restriction of $\theta$ to $\mathfrak{A}_{L/K}$ is a bijection onto $\mathcal{O}_{L}$
and so $\varepsilon$ is a free generator of $\mathcal{O}_{L}$ over 
$\mathfrak{A}_{L/K}$. 
Furthermore, $\theta$ restricted to $\mathcal{O}_{K}[G][\pi_{K}^{-1}\Tr_{G_{0}}]$ also has image 
$\mathcal{O}_{L}$, and so injectivity of $\theta$ shows that in fact 
$\mathfrak{A}_{L/K}=\mathcal{O}_{K}[G][\pi_{K}^{-1}\Tr_{G_{0}}]$.
\end{proof}

\bibliography{WeakRamLocalGMS_Bib}{}
\bibliographystyle{amsalpha}

\end{document}